\documentclass [10pt,oneside]{amsart}
\usepackage{amsmath}

\usepackage{multicol}
\usepackage {amsfonts}
\usepackage {amsmath}
\usepackage {amsthm}
\usepackage {amssymb}
\usepackage {framed}
\usepackage {amsxtra}
\usepackage {enumerate}
\usepackage {graphicx}
\usepackage{color}
\usepackage{graphicx}
\usepackage{graphicx}
\usepackage{wrapfig}
\usepackage{lscape}
\usepackage{rotating}
\usepackage{epstopdf}
\usepackage{pdflscape}
\usepackage{setspace}
\usepackage{adjustbox}
\usepackage{caption}

\usepackage{pgf,tikz,pgfplots}
\pgfplotsset{compat=1.15}
\usepackage{mathrsfs}
\usetikzlibrary{arrows}
\definecolor{uuuuuu}{rgb}{0.26666666666666666,0.26666666666666666,0.26666666666666666}
\definecolor{ududff}{rgb}{0.30196078431372547,0.30196078431372547,1.}
\definecolor{zzttqq}{rgb}{0.6,0.2,0.}
\definecolor{ududff}{rgb}{0.30196078431372547,0.30196078431372547,1.}

\makeatletter

\theoremstyle{definition}
\newtheorem{df}{Definition} [section]

\newtheorem{q}{Question}

\theoremstyle{plain}
\newtheorem{thm}[df]{Theorem}

\newtheorem{lemma}[df]{Lemma}
\newtheorem{cor}[df]{Corollary}

\usepackage{xcolor}

\title{On the maximum area of inscribed polygons}
\author{Dan Ismailescu}
\address{Mathematics Department, Hofstra University, Hempstead, NY 11549}
\email{dan.p.ismailescu@hofstra.edu}

\author{Min Jung Kim}
\address{Choate Rosemary Hall, Wallingford, CT 06492}
\email{alee21@choate.edu}

\author{Eric Wang}
\address{Canterbury School, New Milford, CT 06776}
\email{jhlee00502@gmail.com}

\begin{document}

\begin{abstract}
Given a convex $n$-gon $P$ and a positive integer $m$ such that
$3\le m\le n-1$, let $Q$ denote the largest area convex $m$-gon
contained in $P$. We are interested in the minimum value of $\Delta(Q)/\Delta(P)$, the ratio of
the areas of these two polygons. More precisely, given positive integers $n$ and
$m$, with $3 \le m \le n-1$, define
\begin{equation*}
f_n(m)=\min_{P\in \mathcal {P}_n} \max_{Q \subset P,|Q|=m}
\frac{\Delta(Q)}{\Delta(P)}
\end{equation*}
where the maximum is taken over all $m$-gons contained in $P$, and
the minimum is taken over $\mathcal{P}_n$, the entire class of
convex $n$-gons. The values of $f_4(3)$, $f_5(4)$ and $f_6(3)$ are known.
In this paper we compute the values of $f_5(3)$, $f_6(5)$ and $f_6(4)$.

In addition, we prove that for all $n\ge 6$ we have
\begin{equation*}
\frac{4}{n}\cdot\sin^2\left(\frac{\pi}{n}\right)\le 1-f_n(n-1)\le
\min\left(\frac{1}{n}, \frac{4}{n}\cdot\sin^2\left(\frac{2\pi}{n}\right)\right).
\end{equation*}
These bounds can be used to improve the known estimates for $f_n(m)$.
\end{abstract}

\maketitle

\section{\bf Introduction}
In 1940, S\`{a}s \cite{SAS40} proved the following:
\begin{thm}
Let $K$ be a compact convex body in the plane and let $P_m$ be the
largest area $m$-gon contained in $K$. Then
\begin{equation}\label{sas}
\frac{\Delta(P_m)}{\Delta(K)}\ge \frac{m}{2\pi}\cdot
\sin\left(\frac{2\pi}{m}\right)
\end{equation}
where equality holds if and only if $K$ is an ellipse.
\end{thm}
\noindent Throughout the paper $\Delta(.)$ denotes the
area. At about the same time, Fejes T\'{o}th \cite{LF39} proved an analogous theorem
dealing with $m$-gons containing $K$.
A few years later, L\'{a}z\'{a}r \cite {lazar47} proved a result
involving both polygons that contain $K$ and polygons that are
contained in $K$. The following year, John \cite {J48}
published one of the major breakthroughs in the field of
approximation of convex bodies. John's ellipsoid theorem states that
every $d$-dimensional convex body $K$ lies between two concentric
homothetic ellipsoids, whose ratio is no greater than $d$.

The results mentioned above have various applications to packing and
covering problems as well as in the design of numerous geometric
algorithms - see \cite {ABCG07, T53, TK93, SFRW90, Y87}.

In 1992, motivated by a problem in robot motion planning, Fleischer,
Mehlhorn, Rote, Welzl, and Yap \cite {FMRWY92} raised the following
version of S\`{a}s' problem.
\begin{q}
Given a convex $n$-gon $P$ and a positive integer $m$ such that
$3\le m\le n-1$, let $Q$ denote the largest area convex $m$-gon
contained in $P$. How small can $\Delta(Q)/\Delta(P)$, the ratio of
the areas of these two polygons, be?
\end{q}
More precisely, given positive integers $n$ and $m$, with $3\le m\le
n-1$, define
\begin{equation*}
f_n(m)=\min_{P\in \mathcal {P}_n} \max_{Q \subset P,|Q|=m}
\frac{\Delta(Q)}{\Delta(P)}
\end{equation*}
where the maximum is taken over all $m$-gons contained in $P$, and
the minimum is taken over $\mathcal{P}_n$, the entire class of
convex $n$-gons.
It is easy to see that one can restrict the vertices of $Q$ to an $m$-subset of the vertices of $P$.

There are just a few known values of $f_n(m)$. It
is easy to see that $f_4(3)=\frac{1}{2}$. Indeed, dividing a convex
quadrilateral by a diagonal, we get two triangles at least one of
which has area greater or equal than half of the area of the
quadrilateral. On the other hand, the case of parallelograms shows
that the value $1/2$ cannot be replaced by a larger one.

Du and Ding \cite{DD06, DD08} proved that $f_5(4)=(5+\sqrt{5})/10$, while Fleischer et al. \cite{FMRWY92} showed that $f_6(3)=4/9$.
To the best of our knowledge, these are the only exact values of $f_n(m)$ known at this time.

In this paper we show that
\begin{equation}\label{6465}
f_5(3)=\frac{1}{\sqrt{5}},\,\, f_6(4)=\frac{2}{3},\,\, \text{and}\,\, f_6(5)=\frac{5}{6}.
\end{equation}
Moreover, we find the exact order of magnitude of $f_n(n-1)$ by showing that for every $n\ge 6$
\begin{equation}\label{gn}
\frac{4}{n}\cdot\sin^2\left(\frac{\pi}{n}\right)\le 1-f_n(n-1)\le
\min\left(\frac{1}{n},\frac{4}{n}\cdot\sin^2\left(\frac{2\pi}{n}\right)\right).
\end{equation}
\noindent We use this inequality to prove that $f_n(m)\ge m/n$ for all $n-1\ge m\ge 5$.

\subsection{\bf The Main Technique} Throughout the entire paper we
use the \emph{outer product} of two vectors to express areas. This
operation, also known as \emph{exterior product}, is defined as follows:

For any two vectors $\mathbf{v}=(a,
\, b)$ and $\mathbf{u}=(c,\,d)$, the outer product of $\mathbf{v}$ and $\mathbf{u}$ be given by
$$\mathbf{v}\wedge \mathbf{u}:=\frac{1}{2}\cdot (ad-bc).$$

It is easy to see that the outer product represents the \emph{signed
area} of the triangle determined by the vectors $\mathbf{v}$ and
$\mathbf{u}$, where the $\pm$ sign depends on whether the angle
between $\mathbf{v}$ and $\mathbf{u}$ - measured in the
counterclockwise direction from $\mathbf{v}$ towards $\mathbf{u}$ -
is smaller than or greater than $180^{\circ}$.

The following properties of the outer product are simple
consequences of the definition and are going to be used extensively
in the remaining part of the paper.
\begin{eqnarray*}
1.&& \mathbf{v}\wedge \mathbf{u} =- \mathbf{u}\wedge\mathbf{v},\quad \text{
anticommutativity. In particular,} \quad \mathbf{v}\wedge\mathbf{v}=0.\\
2.&& (\alpha \mathbf{v}+\beta \mathbf{u})\wedge \mathbf{w}=\alpha (\mathbf{v}\wedge\mathbf{w})
+\beta(\mathbf{u}\wedge\mathbf{w}), \quad \text{linearity}.
\end{eqnarray*}

\section{\bf Setup for pentagons}

In this section we are going to prove that $f_5(4)=(5+\sqrt{5})/10$
and $f_5(3)={1}/{\sqrt{5}}$. As mentioned before, the first result has
already been proved in \cite {DD06}, \cite {DD08}.
We believe that our approach leads to simpler, more
transparent proofs.

Let $ABCDE$ be an arbitrary convex pentagon.
After an eventual relabeling of the vertices we may assume that
\begin{equation}\label{ABC}
\Delta(DEA)=\min\{\Delta(ABC),\, \Delta(BCD),\, \Delta(CDE),\,
\Delta(DEA),\, \Delta(EAB)\}
\end{equation}
In the literature, the triangles formed by three consecutive
vertices of a convex polygon are sometimes called \emph{ears}.
Assumption (\ref {ABC}) above fixes the ear of least area. Denote
the intersection of $AC$ and $BD$ by $O$. Then define $\mathbf{u}
= \overrightarrow{OD}$, $\mathbf{v} = \overrightarrow{OA}$. After
an appropriate scaling, we may assume that $\mathbf{u}\wedge
\mathbf{v} = \Delta(AOD) = 1$. Since $A$, $O$, and $C$ are
collinear and $D$, $O$, and $B$ are collinear, we can write
$\overrightarrow{BO} = a \cdot \overrightarrow{OD} = a \mathbf{u}$
and $\overrightarrow{CO} = b \cdot \overrightarrow{OA} = b
\mathbf{v}$, with $a,\ b > 0$ (see figure 1).

\vspace{0.5cm}
\begin{tikzpicture}[line cap=round,line join=round,>=triangle 45,x=1.0cm,y=1.0cm,scale=1.9]
\clip(2.3,2.5) rectangle (9.5,8);
\draw [line width=2.pt] (4.,6.)-- (8.5,3.5);
\draw [line width=2.pt] (4.5,3.5)-- (9.,6.);
\draw [->,line width=2.pt] (6.5,4.611111111111111) -- (7.970540072719675,5.428077818177597);
\draw [->,line width=2.pt] (6.5,4.611111111111111) -- (5.338291767194487,5.25650457378084);
\draw (7.841687420858986,5.2) node[anchor=north west] {\Large{$\mathbf{u}$}};
\draw [line width=2.pt] (4.5,3.5)-- (8.5,3.5);
\draw (5.8,3.4) node[anchor=north west] {\Large{$a\mathbf{u}-b\mathbf{v}$}};
\draw [line width=2.pt] (4.,6.)-- (4.5,3.5);
\draw [line width=2.pt] (4.,6.)-- (6.5,7.5);
\draw [line width=2.pt] (9.,6.)-- (6.5,7.5);
\draw (6.103004519809618,6.95) node[anchor=north west] {\rotatebox{90.0}{\Large{$\mathbf{w}=c\mathbf{u}+d\mathbf{v}$}}};
\draw [line width=2.pt] (8.5,3.5)-- (9.,6.);
\draw [line width=2.pt] (6.5,4.611111111111111)-- (6.5,7.5);
\draw (5.213207976331412,5) node[anchor=north west] {\Large{$\mathbf{v}$}};
\draw (5.458669091773675,4) node[anchor=north west] {\Large{$a\mathbf{u}$}};
\draw (7.074621435101912,4.1) node[anchor=north west] {\Large{$b\mathbf{v}$}};
\draw [->,line width=2.pt] (4.5,3.5) -- (5.540673909390832,4.078152171883795);
\draw [->,line width=2.pt] (8.5,3.5) -- (7.523584905660377,4.0424528301886795);
\draw [->,line width=2.pt] (8.5,3.5) -- (8.738646343211737,4.693231716058685);
\draw [->,line width=2.pt] (4.,6.) -- (4.28175863702411,4.591206814879449);
\draw [->,line width=2.pt] (6.5,7.5) -- (5.247730178457384,6.748638107074431);
\draw [->,line width=2.pt] (9.,6.) -- (7.736615939327128,6.758030436403724);
\draw (3.6586209118637414,5.15) node[anchor=north west]
{\rotatebox{100.0}{\Large{$-a\mathbf{u}-\mathbf{v}$}}};
\draw (8.711028871383672,5.325315190586997) node[anchor=north west] {\rotatebox{80.0}{\Large{$\mathbf{u}+b\mathbf{v}$}}};
\draw (7.36099273645122,7.514010136613849) node[anchor=north west] {\rotatebox{-29.999999999999996}{\Large{$(c-1)\mathbf{u}+d\mathbf{v}$}}};
\draw (4.282501246946162,7.75) node[anchor=north west] {\rotatebox{29.999999999999996}{\Large{$-c\mathbf{u}+(1-d)\mathbf{v}$}}};
\draw [->,line width=2.pt] (4.5,3.5) -- (6.5,3.5);
\draw [->,line width=2.pt] (6.5,4.611111111111111) -- (6.5,6.);
\draw (4.05749522445742,3.5) node[anchor=north west] {\LARGE{$B$}};
\draw (8.475795302418168,3.5) node[anchor=north west] {\LARGE$C$};
\draw (9.079220544547066,6.3685249312166174) node[anchor=north west] {\LARGE$D$};
\draw (6.348465635251882,8.066297646358942) node[anchor=north west] {\LARGE$E$};
\draw (6.32801054229836,4.5) node[anchor=north west] {\LARGE$O$};
\draw (3.5768005400496534,6.450345303030706) node[anchor=north west] {\LARGE$A$};
\begin{scriptsize}
\draw [fill=ududff] (4.,6.) circle (2pt);
\draw [fill=ududff] (4.5,3.5) circle (2pt);
\draw [fill=ududff] (8.5,3.5) circle (2pt);
\draw [fill=ududff] (9.,6.) circle (2pt);
\draw [fill=ududff] (6.5,7.5) circle (2pt);
\draw [fill=uuuuuu] (6.5,4.611111111111111) circle (2.0pt);
\end{scriptsize}
\end{tikzpicture}
\vspace{-1.5cm}
\label{fig1}
\begin{figure}[h!]
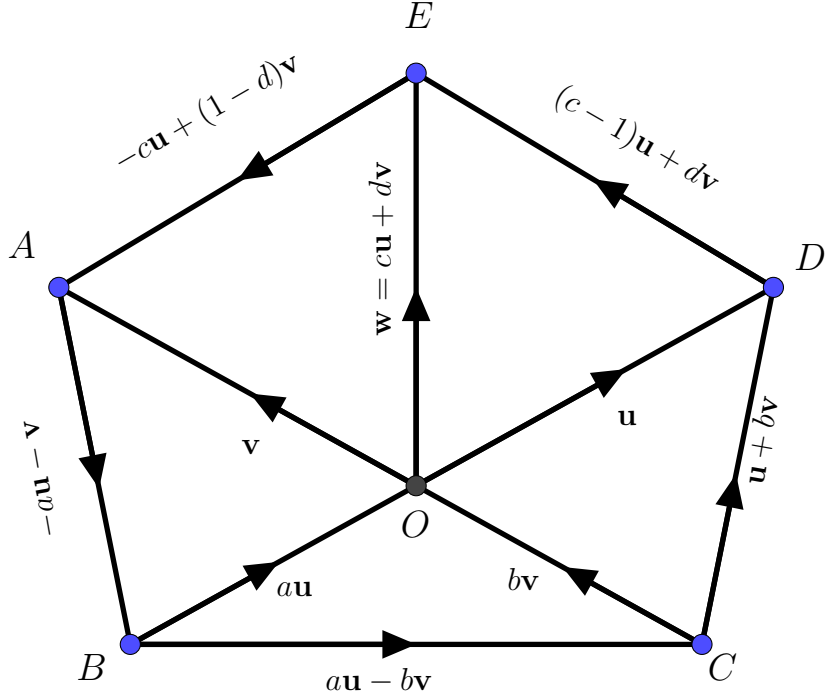

\caption{Defining a pentagon }
\end{figure}

Without loss of generality we may assume that $a\leq b$.
\noindent Using the triangle rule, we obtain that
$\overrightarrow{AB} = -a\mathbf{u}-\mathbf{v}$,
$\overrightarrow{BC} = a\mathbf{u} - b\mathbf{v}$, and
$\overrightarrow{CD} = \mathbf{u}+b\mathbf{v} $. We know that
every vector in the plane can be written as a linear combination of
any two independent vectors.

Set $\overrightarrow{OE} = \mathbf{w} = c \mathbf{u} + d \mathbf{v}$, with $c, d>0$.
It follows that $\overrightarrow{DE} = \mathbf{w} - \mathbf{u}$
and $\overrightarrow{AE} = \mathbf{v} - \mathbf{w}$. We have
\begin{eqnarray*}
\Delta(ODE) &=& \mathbf{u}\wedge \mathbf{w} =
\mathbf{u}\wedge (c\mathbf{u} + d\mathbf{v}) = c(\mathbf{u}\wedge \mathbf{u}) + d (\mathbf{u}\wedge\mathbf{v}) = d,\\
\Delta(OAE)&=& \mathbf{w}\wedge \mathbf{v} =
(c\mathbf{u} + d\mathbf{v})\wedge \mathbf{v} = c(\mathbf{u}\wedge \mathbf{v}) + d (\mathbf{v}\wedge \mathbf{v}) = c.
\end{eqnarray*}

After similar calculations, we can write the areas of various
triangles in pentagon $ABCDE$ in terms of the positive constants
$a,\,b,\,c,\,d$ as shown below:
\begin{eqnarray*}
\Delta(OAB) &=& -\mathbf{v}\wedge a\mathbf{u} = a,\\
\Delta(OBC) &=& -a\mathbf{u}\wedge -b\mathbf{v} =ab,\\
\Delta(OCD) &=& -b\mathbf{v}\wedge (b\mathbf{v}+ \mathbf{u}) =b.
\end{eqnarray*}

We can now compute the total area of the pentagon

$\Delta(ABCDE) = \Delta(OAB) + \Delta(OBC) + \Delta(OCD) +
\Delta(ODE) + \Delta(OEA)$, that is,
\begin{equation}\label{ABCDE}
\Delta(ABCDE) = a+b+c+d+ab.
\end{equation}

Next, we compute the areas of the ears of the pentagon.
\begin{eqnarray}\label{laterlemma}
\Delta(ABC) &=& \overrightarrow{AB} \wedge \overrightarrow{BC}= (-a\mathbf{u}-\mathbf{v})\wedge (a\mathbf{u} - b\mathbf{v}) = a + ab,\notag\\
\Delta(BCD) &=& \overrightarrow{BC} \wedge \overrightarrow{CD}= (a\mathbf{u}- b\mathbf{v})\wedge (b\mathbf{v} + \mathbf{u}) = b + ba,\notag\\
\Delta(CDE) &=& \overrightarrow{CD} \wedge \overrightarrow{DE}= (b\mathbf{v}+\mathbf{u})\wedge ((c-1)\mathbf{u} + d\mathbf{v}) = b + d - bc,\\
\Delta(DEA) &=& \overrightarrow{DE} \wedge \overrightarrow{EA}= ((c-1)\mathbf{u}+ d\mathbf{v})\wedge (-c\mathbf{u} + (1-d)\mathbf{v}) = c +d -1,\notag\\
\Delta(EAB) &=& \overrightarrow{EA} \wedge \overrightarrow{AB}= (-c\mathbf{u}+ (1-d)\mathbf{v})\wedge (-a\mathbf{u} - \mathbf{v}) = a + c  -ad.\notag
\end{eqnarray}
Translating assumption (\ref{ABC}) in terms of $a$, $b$, $c$, and $d$, we obtain
\begin{eqnarray*}
\Delta(DEA) &\leq& \Delta(ABC)\iff c + d - 1 \leq a + ab.
\iff c + d \leq 1+a+ab\\
\Delta(DEA) &\leq& \Delta(CDE) \iff c+d-1 \leq b+d-
bc\iff (c-1)(1+b) \leq 0\iff c\leq 1,\\
\Delta(DEA) &\leq& \Delta(EAB) \iff c+d-1 \leq a + c -
ad\iff (d-1)(1+a) \leq 0\iff d\leq 1.
\end{eqnarray*}
Therefore,
\begin{equation}\label{cdineq}
c\le 1, \,\,d\le 1,\,\,\text{and}\,\,c+d\le 1+a+ab
\end{equation}

We are now in position to find the values of $f_5(4)$ and $f_5(3)$.

\subsection{\bf Large Quadrilaterals in Pentagons}

\begin{thm}(\cite {DD06, DD08}) \label{f54}
$f_5 (4) = \frac{5 + \sqrt{5}}{10}.$
\end{thm}

\begin{proof}
It is easy to check that $f_5(4) \leq \frac{5+\sqrt{5}}{10}$ as the
case of the regular pentagon shows. From (\ref {ABC}) it follows
that $ABCD$ is the largest quadrilateral contained in the pentagon
$ABCDE$. It remains to prove the opposite inequality, that is,
\begin{equation}\label{q1}
\frac{\Delta(ABCD)}{\Delta(ABCDE)} \geq \frac{5 + \sqrt{5}}{10}.
\end{equation}
Since $\Delta(ABCDE)=\Delta(ABCD)+\Delta(DEA)$, (\ref {q1}) is
equivalent to
\begin{equation}\label{q2}
\frac{\Delta(DEA)}{\Delta(ABCDE)} \leq \frac{5 - \sqrt{5}}{10} =
\frac{2}{5+\sqrt{5}}.
\end{equation}
Since $\Delta(DEA)=c+d-1$ and $\Delta(ABCDE)=a+b+c+d+ab$,
relation (\ref{q2}) translates to $(5+\sqrt{5})(c+d-1) \leq
2a+2b+2c+2d+2ab$, that is,
\begin{equation}\label{T1}
(3+\sqrt{5})(c+d)\leq 5+\sqrt{5}+2a+2b+2ab.
\end{equation}

To prove \eqref{T1} we will proceed in two cases.

\emph{Case 1:} $a \leq \frac{\sqrt{5}- 1} {2}$.

\noindent Using (\ref{cdineq}), we have $c+d\le 1+a+ab$, so it would be sufficient to show that
\begin{align*}
(3+\sqrt{5})(1+a+ab) &\leq 5+\sqrt{5}+2a+2b+2ab \iff\\
\iff(1+\sqrt{5})(a+ab) &\leq 2+2b
\iff(1+\sqrt{5})a\leq 2 \iff a\leq
\frac{\sqrt{5}-1}{2}.
\end{align*}

\smallskip

\emph{Case 2:} $a \geq \frac{\sqrt{5} - 1} {2}$.

\noindent In this case, using (\ref{cdineq}) we have $c+d\le 2$, so it would
suffice to prove that
\begin{equation*}
2(3+\sqrt{5})\leq 5+\sqrt{5}+2a+2b+2ab \stackrel{a\le
b}{\Longleftarrow} 1+\sqrt{5}\leq 4a+2a^{2}
\stackrel{a>0}{\iff} a\geq \frac{\sqrt{5}-1} {2}.
\end{equation*}
This completes the proof of Theorem \ref{f54}.
\end{proof}

\subsection{\bf Large Triangles in Pentagons}

We will now compute the areas of three of the remaining five triangles determined from the vertices of $ABCDE$ - refer to figure 1.
\begin{align}\label{mama}
\Delta(ABD) &= \Delta(ABCDE)-\Delta(DEA)-\Delta(BCD)=a+1\notag\\
\Delta(ACD) &= \Delta(ABCDE)-\Delta(DEA)-\Delta(ABC)=b+1\\
\Delta(BCE) &= \Delta(ABCDE)-\Delta(BAE)-\Delta(CDE)=ab+ad+bc.\notag
\end{align}
Recall that $\Delta(BCD) = ab+b$.

\begin{thm}\label{f53}
$f_5(3)=\frac{1}{\sqrt{5}}$.
\end{thm}

\begin{proof}
It is easy to verify that $f_5(3) \leq \frac{1}{\sqrt{5}}$ as
the case of the regular pentagon shows. It would be sufficient to
show that, with the notations above,
\begin{equation}\label{f53b}
\frac{\max\{\Delta(ACD),\Delta(BCE),\Delta(BCD)\}}{\Delta(ABCDE)}
\geq \frac{1}{\sqrt{5}}
\end{equation}

\emph{Case 1:} $a \leq \frac{\sqrt{5} - 1} {2}$.

In this case, we have
\begin{align*}
\frac{\Delta(ACD)}{\Delta(ABCDE)} \geq \frac{1}{\sqrt{5}}&
\stackrel{(\ref{mama})}{\iff} \sqrt{5} (b+1) \geq
a+b+(c+d)+ab \stackrel{(\ref{cdineq})}{\Longleftarrow}\\
&\Longleftarrow \sqrt{5} (b+1)\geq a+b+(1+a+ab)+ab \iff\\
&\iff(\sqrt{5} -1)(b+1) \geq 2a(1+b) \iff a \leq
\frac{\sqrt{5} -1}{2}.
\end{align*}

This settles Case 1.

\bigskip

\emph{Case 2:} $\frac{\sqrt{5}-1}{2} \leq a \leq \frac{\sqrt{5} +
1}{2}$.

\emph{Subcase 2(i):} $\Delta(ACD) \geq \Delta(BCE)$.

In this case, $b+1 \geq ab+ad+bc \stackrel{a\le b}{\Longrightarrow}
b+1 \geq ab+ad+ac \Longleftrightarrow c+d\leq\frac{b+1}{a} - b$.

\begin{align*}
\text{Then},\,\, \frac{\Delta(ACD)}{\Delta(ABCDE)} \geq \frac{1}{\sqrt{5}}
&\iff \sqrt{5}(b+1) \geq a+b+(c+d)+ab \Longleftarrow\\
&\Longleftarrow \sqrt{5}(b+1) \geq a+b+\left(\frac{b+1}{a}-b\right)+ab\iff\\
&\iff \sqrt{5}(b+1) \geq a(b+1) + \frac{b+1}{a}\iff\\
&\iff \sqrt{5}\geq a+\frac{1}{a}
\iff\frac{\sqrt{5}-1}{2} \leq a \leq \frac{\sqrt{5} +
1}{2}.
\end{align*}

\emph{Subcase 2(ii):} $\Delta(BCE) \geq \Delta(ACD)$.

In this case, $ab+ad+bc \geq 1+b$ which gives
\begin{equation}\label{subcase2b}
d \geq \frac{1+b-bc}{a}-b.
\end{equation}
\begin{align*}
\text{Then,}\,\,\frac{\Delta(BCE)}{\Delta(ABCDE)} \geq \frac{1}{\sqrt{5}}&
\iff \sqrt{5}(ab+ad+bc) \geq a+b+c+d+ab \iff\\
&\iff (\sqrt{5}-1)ab+d(\sqrt{5}a-1)+\sqrt{5}bc \geq a+b+c
\stackrel{(\ref{subcase2b})}{\Longleftarrow}\\
&\Longleftarrow(\sqrt{5}-1)ab+\left(\frac{1+b-bc}{a}-b\right)(\sqrt{5}a-1)+\sqrt{5}bc
\geq a+b+c\\
&\iff \sqrt{5}(1+b)-\frac{1+b}{a} \geq
a(1+b)+c\left(1-\frac{b}{a}\right) \stackrel{a\le b}{\Longleftarrow}\\
&\Longleftarrow \sqrt{5}(1+b) \geq a(1+b)+\frac{1+b}{a} \iff\\
&\iff \sqrt{5}\geq a+\frac{1}{a} \iff \frac{\sqrt{5}-1}{2} \leq
a \leq \frac{\sqrt{5} + 1}{2}.
\end{align*}
This ends the proof of Case 2.

\emph{Case 3:} $a \geq \frac{\sqrt{5} + 1}{2}$.

In this case we have
\begin{align*}
\frac{\Delta(BCD)}{\Delta(ABCDE)} \geq \frac{1}{\sqrt{5}}&
\iff \sqrt{5}(ab+b) \geq a+b+c+d+ab \stackrel{(\ref{cdineq})}{\Longleftarrow}\\
&\Longleftarrow \sqrt{5}b(a+1) \geq a+b+2+ab \iff
(\sqrt{5}-1)b(a+1)\ge a+2 \stackrel{a\le b}{\Longleftarrow}\\
&\Longleftarrow (\sqrt{5}-1)a(a+1)\geq a+2 \Longleftarrow a(a+1) \geq a+2
\iff\\
&\iff a^{2} \geq 2 \iff a \geq \sqrt{2}
\Longleftarrow a\geq \frac{\sqrt{5}+1}{2} > \sqrt{2}.
\end{align*}
This completes the proof of Theorem \ref{f53}.
\end{proof}

\section{\bf Setup for hexagons}

In this section we present proofs of two new results: $f_6(5)=5/6$
and $f_6(4)=2/3$.

Let $ABCDEF$ be an arbitrary convex hexagon. Suppose that the long
diagonals, $AD$, $BE$, and $CF$ are not concurrent. If these
diagonals do have a common point, then perturb the position of one
of the vertices by an arbitrarily small amount so that the diagonals
are not concurrent anymore. By continuity, any inequality which is
valid in the latter case is also valid in the former.
Let $M=AD\cap BE,\,N=AD\cap CF,\,$ and $P=CF\cap BE$. Denote
$\mathbf{u}=\overrightarrow{MN},
\mathbf{v}=\overrightarrow{MP}$ as shown in figure 2.

It follows that
$\overrightarrow{NP}=\mathbf{v}-\mathbf{u}$.
After an appropriate scaling we may assume that $\Delta(MNP)=
\mathbf{u}\wedge\mathbf{v} = 1$. Since $A,\,M,\,N,\,D$ are
collinear, $\overrightarrow{AM}=a\mathbf{u},
\overrightarrow{ND}=d\mathbf{u}$ with $a,\,d >0$.

Similarly,
$\overrightarrow{BM}=b\mathbf{v},
\overrightarrow{CN}=c(\mathbf{v}-\mathbf{u}),
\overrightarrow{PE}=e\mathbf{v},
\overrightarrow{PF}=f(\mathbf{v}-\mathbf{u})$ with
$b,\,c,\,e,\,f$ being positive constants. By the symmetry of the figure, we may assume that $a=\min \{a,\,b,\,c,\,d,\,e,\,f\}$.

\begin{tikzpicture}[line cap=round,line join=round,>=triangle 45,x=1.0cm,y=1.0cm,scale=2.3]\label{fig2}
\clip(2.7,3.373196063971163) rectangle (11.22571750762268,10.100829735319762);
\draw (6.306587511367731,5.6) node[anchor=north west] {\Large{$\mathbf{u}$}};
\draw [line width=2.pt] (4.6,4.4)-- (8.424202850376,4.366802955064267);
\draw [line width=2.pt] (4.,6.)-- (4.6,4.4);
\draw [line width=2.pt] (4.,6.)-- (4.6709579020016845,8.43816618841721);
\draw [line width=2.pt] (8.424202850376,4.366802955064267)-- (9.2,5.4);
\draw (5.627866335337738,6.538077495385155) node[anchor=north west] {\Large{$\mathbf{v}$}};
\draw (6.793818907903349,5.064147000183781) node[anchor=north west] {\Large{$c(\mathbf{v}-\mathbf{u})$}};
\draw (5.185315232809618,5.199784775815809) node[anchor=north west] {\Large{$b\mathbf{v}$}};
\draw [->,line width=2.pt] (8.424202850376,4.366802955064267) -- (8.794485156039729,4.85994034824646);
\draw [->,line width=2.pt] (4.,6.) -- (4.338110364428932,5.098372361522848);
\draw [->,line width=2.pt] (4.6709579020016845,8.43816618841721) -- (4.334869769120782,7.216869413353026);
\draw [->,line width=2.pt] (4.6,4.4) -- (6.512101425188,4.383401477532134);
\draw (4.4167011708947825,4.3) node[anchor=north west] {\Large{$B$}};
\draw (8.3,4.3) node[anchor=north west] {\Large{$C$}};
\draw (9.3,5.6) node[anchor=north west] {\Large{$D$}};
\draw (7.8,9.449768412286026) node[anchor=north west] {\Large{$E$}};
\draw (5.3,6.212546833868287) node[anchor=north west] {\Large{$M$}};
\draw (3.6,6.2) node[anchor=north west] {\Large{$A$}};
\draw [line width=2.pt] (4.6709579020016845,8.43816618841721)-- (7.976090584343801,9.033849711965273);
\draw [line width=2.pt] (9.2,5.4)-- (7.976090584343801,9.033849711965273);
\draw [line width=2.pt] (4.6,4.4)-- (7.976090584343801,9.033849711965273);
\draw [line width=2.pt] (4.6709579020016845,8.43816618841721)-- (8.424202850376,4.366802955064267);
\draw [line width=2.pt] (4.,6.)-- (9.2,5.4);
\draw [->,line width=2.pt] (5.628788669193889,5.812062845862244) -- (6.631873115138083,5.696322332868683);
\draw [->,line width=2.pt] (5.628788669193889,5.812062845862244) -- (6.109934109379255,6.472458532445908);
\draw (4.4167011708947825,8.898174791382445) node[anchor=north west] {\Large{$F$}};
\draw (6.6,6.5) node[anchor=north west] {\Large{$\mathbf{v}-\mathbf{u}$}};
\draw (6.1,7.2) node[anchor=north west] {\Large{$P$}};
\draw (7.2,6.013611429607979) node[anchor=north west] {\Large{$N$}};
\draw [->,line width=2.pt] (6.274653611265102,6.69854411791818) -- (7.068876069511733,7.7886533485322715);
\draw (7.85285815357287,5.9) node[anchor=north west] {\Large{$d\mathbf{u}$}};
\draw [->,line width=2.pt] (4.,6.) -- (5.112119177905328,5.87167855639554);
\draw (4.8,6.23) node[anchor=north west] {\Large{$a\mathbf{u}$}};
\draw (5.251274350787123,8.042329332932009) node[anchor=north west] {\Large{$f(\mathbf{v}-\mathbf{u})$}};
\draw (7.05,7.7) node[anchor=north west] {\Large{$e\mathbf{v}$}};
\draw [->,line width=2.pt] (7.266019752765538,5.623151566988591) -- (8.0878900225277,5.528320382016035);
\draw [->,line width=2.pt] (9.2,5.4) -- (8.665857875962951,6.985895311165562);
\draw [->,line width=2.pt] (7.976090584343801,9.033849711965273) -- (6.333611282795938,8.737825935695126);
\draw [->,line width=2.pt] (4.6,4.4) -- (5.140613067778788,5.142017918603423);
\draw [->,line width=2.pt] (6.274653611265102,6.69854411791818) -- (5.472805756633393,7.568355153167695);
\draw [->,line width=2.pt] (7.266019752765538,5.623151566988591) -- (6.77033668201532,6.160847842453386);
\draw [->,line width=2.pt] (8.424202850376,4.366802955064267) -- (7.845111301570769,4.994977261026429);
\draw (5.556058486203833,9.314130636653998) node[anchor=north west] {\Large{$\rotatebox{10.0}{$ -f\mathbf{u}+(f-e)\mathbf{v}$}$}};
\draw (8.5,5.389677661700649) node[anchor=north west] {\Large{$\rotatebox{56.0}{$(d-c)\mathbf{u}+c\mathbf{v}$}$}};
\draw (5.565101004579302,4.331703011770829) node[anchor=north west] {\Large{$\rotatebox{0.0}{ $ (c+1)\mathbf{u}+(b-c)\mathbf{v}$ }$}};
\draw (3.765107549991195,5.651910694589237) node[anchor=north west] {\Large{$\rotatebox{-70.0}{$a\mathbf{u}-b\mathbf{v}$}$}};
\draw (3.5932997008572905,8.542326060226053) node[anchor=north west] {\Large{$\rotatebox{75}{$(f-a)\mathbf{u}-(f+1)\mathbf{v} $}$}};
\draw (8.36,8.355623688854331) node[anchor=north west] {\Large{$\rotatebox{-72.0}{$ -(d+1)\mathbf{u}+(e+1)\mathbf{v}$}$}};
\begin{scriptsize}
\draw [fill=ududff] (4.,6.) circle (2pt);
\draw [fill=ududff] (4.6,4.4) circle (2pt);
\draw [fill=ududff] (8.424202850376,4.366802955064267) circle (2pt);
\draw [fill=ududff] (9.2,5.4) circle (2pt);
\draw [fill=ududff] (4.6709579020016845,8.43816618841721) circle (2pt);
\draw [fill=ududff] (7.976090584343801,9.033849711965273) circle (2pt);
\draw [fill=uuuuuu] (6.274653611265102,6.69854411791818) circle (2.0pt);
\draw [fill=uuuuuu] (5.628788669193889,5.812062845862244) circle (2.0pt);
\draw [fill=uuuuuu] (7.266019752765538,5.623151566988591) circle (2.0pt);
\end{scriptsize}
\end{tikzpicture}
\vspace{-2cm}
\begin{figure}[h!]
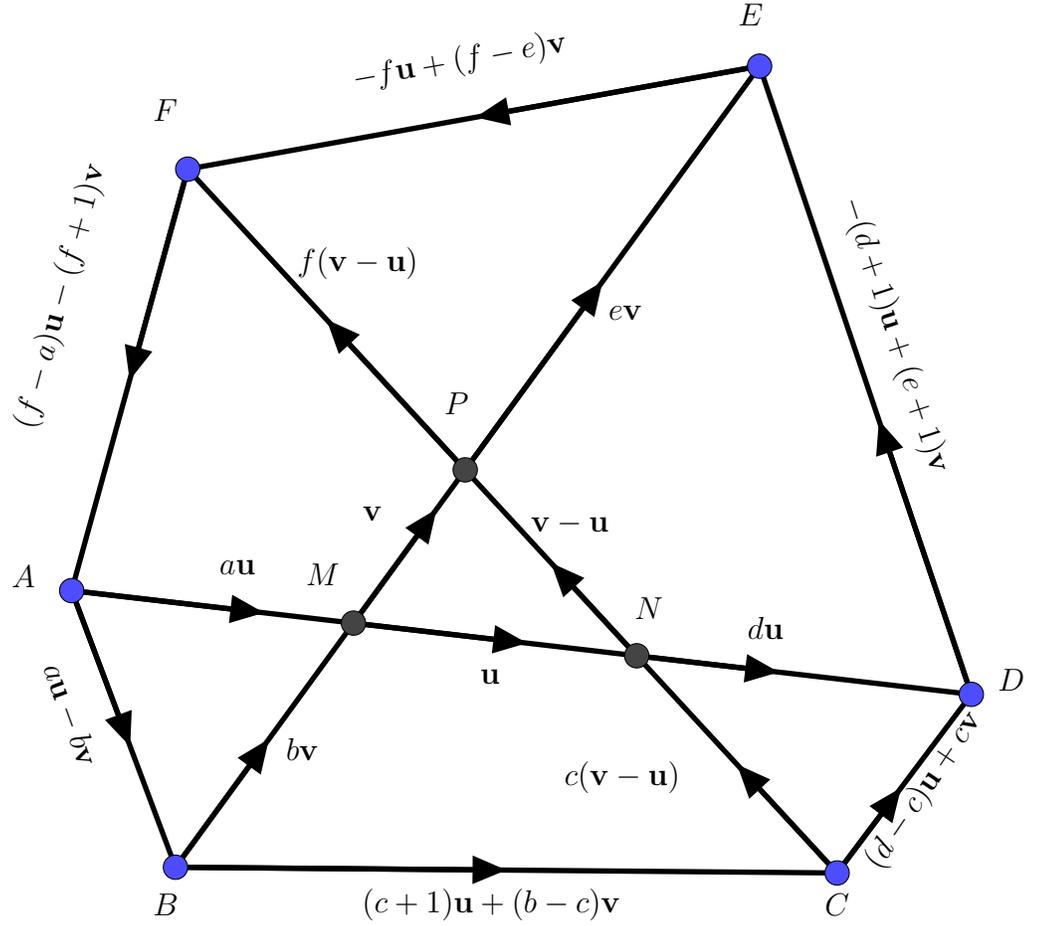

  \caption{The setup for hexagons}
\end{figure}

Using the outer product, we obtain the following area formulas.
\begin{eqnarray*}
\Delta(ANF)&=&(a+1)\mathbf{u}\wedge (f+1)(\mathbf{v}-\mathbf{u})=(a+1)(f+1).\\
\Delta(BPC)&=&(b+1)\mathbf{v}\wedge
(c+1)(\mathbf{v}-\mathbf{u})=(b+1)(c+1).\\
\Delta(DME)&=&(d+1)\mathbf{u}\wedge (e+1)(\mathbf{v})=(d+1)(e+1).\\
\Delta(AMB)&=& a\mathbf{u}\wedge b\mathbf{v}=ab.\\
\Delta(CND)&=& d\mathbf{v}\wedge c(\mathbf{v}-\mathbf{u})=cd.\\
\Delta(EPF)&=& e\mathbf{v}\wedge f(\mathbf{v}-\mathbf{u})=ef.
\end{eqnarray*}
We are now in position to compute the area of the hexagon $ABCDEF$.
\begin{eqnarray*}
\Delta(ABCDEF)&=&\Delta(ANF)+\Delta(BPC)+\Delta(DME)+\Delta(AMB)+\\
&+&\Delta(CND)+\Delta(EPF)-2\Delta(MNP)\qquad {\mbox{which implies
that}}
\end{eqnarray*}
\begin{equation}\label{hex}
\Delta(ABCDEF)=1+a+b+c+d+e+f+ab+bc+cd+de+ef+fa.
\end{equation}

\subsection{\bf Large Pentagons in Hexagons}
The following result is optimal. A lower bound of $f_6(5) \geq
\frac{10}{15-\sqrt{5}} \approx 0.78345\ldots$ had been given by Du
and Ding \cite{DD06}.

\begin{thm}\label{f65}
$f_6(5)=\frac{5}{6}.$
\end{thm}

\begin{proof}

It is easy to check that $f_6(5)\leq 5/6$, as the case of the
regular hexagon shows. The proof will be complete if we show that
${\Delta(BCDEF)}/{\Delta(ABCDEF)}\geq {5}/{6}$. Proving this is
equivalent to proving that ${\Delta(FAB)}/{\Delta(ABCDEF)}\leq
{1}/{6}$. From figure 2, we have that
$\overrightarrow{AB}=a\mathbf{u}-b\mathbf{v}$ and
$\overrightarrow{AF}=(f-a)\mathbf{u}-(1+f)\mathbf{v}$. It
follows that
\begin{align*}
\Delta(FAB)= \overrightarrow{AB}\wedge \overrightarrow{AF}=(a\mathbf{u}-b\mathbf{v})\wedge ((a-f)\mathbf{u}+(1+f)\mathbf{v}).
\end{align*}
After simplifying, we obtain
\begin{equation}\label{fab}
\Delta(FAB)=a(1+f)+b(a-f)=a(1+b+f)-bf.
\end{equation}
From the assumption that $a=\min\{a,\,b,\,c,\,d,\,e,\,f\}$ , it
follows that there exist nonnegative numbers $x_i,\, i=1,\ldots,\,5$
such that
\begin{equation*}
b=a+x_1,\,c=a+x_2,\,d=a+x_3,\,e=a+x_4,\,f=a+x_5.
\end{equation*}
Substituting the above equalities into (\ref{hex}) and (\ref{fab}),
we obtain that
\begin{align*}
\Delta(ABCDEF)-6\Delta(FAB)=\,1&+(2a+1)(x_1+x_2+x_3+x_4+x_5)+\\
&+x_1x_2+x_2x_3+x_3x_4+x_4x_5+6x_5x_1.
\end{align*}
Thus, for any hexagon with nonconcurrent long diagonals, $\Delta(ABCDEF)-6\Delta(FAB)>0$.
By continuity, the non-strict inequality holds for all convex hexagons.
This completes the proof of Theorem \ref{f65}.
\end{proof}

\subsection{\bf Large Quadrilaterals in Hexagons}
The following result is optimal. A lower bound of $f_6(5) \geq
\frac{1}{4-\sqrt{5}} \approx 0.566915\ldots$ had been proved by Du
and Ding \cite{DD08}.

\begin{thm}\label{f64}
$f_6(4)=\frac{2}{3}.$
\end{thm}

\begin{proof}
It is easy to check that $f_6(4)\leq 2/3$, as the case of the
regular hexagon shows. It remains to be shown that $f_6(4) \geq
2/3$. Let us define $Q_1:=\Delta(BCDF)$, $Q_2:=\Delta(ACDE)$,
$Q_3:=\Delta(BDEF)$, $Q_4:=\Delta(BCEF)$, and $H:=\Delta(ABCDEF)$.
We will show that at least one of $Q_1/H$, ${Q_2}/{H}$, ${Q_3}/{H}$,
and ${Q_4}/{H}$ is greater than or equal to $2/3$.
To this end, we will show that at least one of $3Q_1-2H$,
$3Q_2-2H$, $3Q_3-2H$, and $3Q_4-2H$ is nonnegative.

First we will compute the areas of all six ears of the hexagon $ABCDEF$.
\begin{align*}
\Delta(ABC)&= \overrightarrow{AB}\wedge \overrightarrow{BC}= (a\mathbf{u}-b\mathbf{v})\wedge
((c+1)\mathbf{u}+(b-c)\mathbf{v})=b(a+c+1)-ac,\\
\Delta(BCD)&=\overrightarrow{BC}\wedge \overrightarrow{CD}=((c+1)\mathbf{u}+(b-c)\mathbf{v})\wedge((-c-d)\mathbf{u}+c\mathbf{v})=c(b+d+1)-bd,\\
\Delta(CDE)&=\overrightarrow{CD}\wedge \overrightarrow{DE}=((-c-d)\mathbf{u}+c\mathbf{v})\wedge
((1-d)\mathbf{u}+(e-1)\mathbf{v})=d(c+e+1)-ce,\\
\Delta(DEF)&=\overrightarrow{DE}\wedge \overrightarrow{EF}=((1-d)\mathbf{u}+(e-1)\mathbf{v})\wedge (-f\mathbf{u}+(f-e)\mathbf{v})=e(d+f+1)-df,\\
\Delta(EFA)&=\overrightarrow{EF}\wedge \overrightarrow{FA}=(-f\mathbf{u}+(f-e)\mathbf{v})\wedge
((f-a)\mathbf{u}+(-f-1)\mathbf{v})=f(e+a+1)-ea,\\
\Delta(FAB)&=\overrightarrow{FA} \wedge \overrightarrow{AB}=((f-a)\mathbf{u}+(-f-1)\mathbf{v})\wedge (a\mathbf{u} - b\mathbf{v})= a(f+b+1)-fb.
\end{align*}
With this information, we now compute $Q_1$, $Q_2$, $Q_3$, and $Q_4$
in terms of $a$, $b$, $c$, $d$, $e$, and $f$.
\begin{eqnarray}
\label{qq1}Q_1=H-\Delta(FAB)-\Delta(DEF)=1+b+c+d+f+bc+cd+bf+df.\\
\label{qq2}Q_2=H-\Delta(ABC)-\Delta(EFA)=1+c+d+e+a+cd+de+ca+ea.\\
\label{qq3}Q_3=H-\Delta(FAB)-\Delta(BCD)=1+d+e+f+b+de+ef+db+fb.\\
\label{qq4}Q_4=H-\Delta(FAB)-\Delta(CDE)=1+b+c+e+f+bc+ef+bf+ce.
\end{eqnarray}
Let us define $T_1:=3Q_1-2H$, $T_2:=3Q_2-2H$, $T_3:=3Q_3-2H$,
$T_4:=3Q_4-2H$. We need to prove that at least one of $T_1$, $T_2$,
$T_3$, and $T_4$ is nonnegative. We split the proof into five different cases, depending on the value
of the second smallest element of $\{a,b,c,d,e,f\}$.


\noindent{\emph {Case 1:}} $a \leq c \leq \min\{b,d,e,f\}$.

\noindent It follows that there exist nonnegative numbers
$x_i,\,i=1,\ldots,\,5$ such that
\begin{equation*}
c=a+x_1, b=a+x_1+x_2, d=a+x_1+x_3, e=a+x_1+x_4, f=a+x_1+x_5.
\end{equation*}
Substituting the above equalities into (\ref{hex}), (\ref{qq2}), (\ref{qq3}), and (\ref{qq4}),
we obtain that
\begin{equation*}
T_2+T_3+T_4=3+(3+6a+3x_1)(x_1+x_4)+3x_1^2+6x_1(x_2+x_5)+3x_2(x_3+2x_5),
\end{equation*}
which is clearly nonnegative. Therefore, at least one of the numbers $T_2$, $T_3$, $T_4$ is nonnegative.


\noindent{\emph {Case 2:}} $a \leq d \leq \min\{b,c,e,f\}$.

\noindent It follows that there exist nonnegative numbers
$x_i,\,i=1,\ldots,\,5$ such that
\begin{equation*}
d=a+x_1, b=a+x_1+x_2, c=a+x_1+x_3, e=a+x_1+x_4, f=a+x_1+x_5.
\end{equation*}
Substituting the above equalities into (\ref{hex}), (\ref{qq1}), (\ref{qq3}), and (\ref{qq4}),
we obtain that
\begin{equation*}
T_1+T_3+T_4=3+(3+6a+6x_1)(2x_1+x_2+x_5)+6x_1(x_2+x_5)+9x_2x_5+3x_3x_4,
\end{equation*}
which is obviously nonnegative. Hence, at least one of the numbers $T_1$, $T_3$, $T_4$ is nonnegative.


\noindent{\emph{Case 3:}} $a \leq e \leq \min\{b,c,d,f\}$.

\noindent In this case there exist nonnegative numbers
$x_i,\,i=1,\ldots,\,5$ such that
\begin{equation*}
e=a+x_1, b=a+x_1+x_2, c=a+x_1+x_3, d=a+x_1+x_4, f=a+x_1+x_5.
\end{equation*}
Substituting the above equalities into (\ref{hex}), (\ref{qq1}), (\ref{qq2}), and (\ref{qq3}),
we obtain that
\begin{equation*}
T_1+T_2+T_4=3+(3+6a+3x_1)(x_1+x_3)+3x_1^2+6x_1(x_2+x_5)+3x_5(2x_2+x_4)
\end{equation*}
which is nonnegative. It follows that at least one of the numbers $T_1$, $T_2$, $T_4$ is nonnegative.


\noindent{\emph{Case 4:}} $a \leq b \leq \min\{c,d,e,f\}$. We investigate four separate subcases.


\noindent{\emph{Subcase 4(i):}} $a \leq b \leq c \leq \min\{d, e, f\}$.

\noindent In this case there exist nonnegative numbers
$x_i,\,i=1,\ldots,\,5$ such that
\begin{equation*}
b=a+x_1, c=a+x_1+x_2, d=a+x_1+x_2+x_3, e=a+x_1+x_2+x_4,
f=a+x_1+x_2+x_5.
\end{equation*}
Substituting the above equalities into (\ref{hex}), (\ref{qq1}), and  (\ref{qq3}), we obtain that
\begin{equation*}
T_1+2T_3=3+(3+6a+6x_1)(2x_1+x_2+x_3+x_5)+6x_1(x_2+x_5)+3x_5(x_2+x_3),
\end{equation*}
which is nonnegative. Thus, at least one of $T_1$ and $T_3$ must be nonnegative.


\noindent{\emph{Subcase 4(ii):}} $a \leq b \leq d \leq \min\{c, e, f\}$.

\noindent In this case there exist nonnegative numbers $x_i, i=1,\ldots,\,5$ such that
\begin{equation*}
b=a+x_1, d=a+x_1+x_2, c=a+x_1+x_2+x_3, e=a+x_1+x_2+x_4,
f=a+x_1+x_2+x_5.
\end{equation*}
Substituting the above equalities into (\ref{hex}), (\ref{qq1}), (\ref{qq3}), and  (\ref{qq4}), we obtain that
\begin{equation*}
T_1+T_3+T_4=3+(3+6a+6x_1)(2x_1+x_2+x_5)+6x_1(x_2+x_5)+3x_2x_5+3x_3x_4
\end{equation*}
which is nonnegative. Hence, at least one of $T_1$, $T_3$, and $T_4$ must be
nonnegative.


\noindent{\emph{Subcase 4(iii):}} $a \leq b \leq e \leq \min\{c, d, f\}$.

\noindent In this case there exist nonnegative numbers
$x_i,\,i=1,\ldots,\,5$ such that
\begin{equation*}
b=a+x_1, e=a+x_1+x_2, c=a+x_1+x_2+x_3, d=a+x_1+x_2+x_4,
f=a+x_1+x_2+x_5.
\end{equation*}
Substituting the above equalities into (\ref{hex}), (\ref{qq1}), (\ref{qq2}), and  (\ref{qq4}), we obtain that
\begin{equation*}
T_1+T_2+T_4=3+(3+6a+3x_1)(x_1+x_2+x_3)+3x_1(x_1+x_2+2x_5)+3x_2x_3+3x_4x_5,
\end{equation*}
which is clearly nonnegative. It follows that at least one of $T_1$, $T_2$, and $T_4$ must be nonnegative.


\noindent{\emph{Subcase 4(iv):}} $a \leq b \leq f \leq \min\{c, d, e\}$.

\noindent In this case there exist nonnegative numbers $x_i, i=1
\ldots 5$ such that
\begin{equation*}
b=a+x_1, f=a+x_1+x_2, c=a+x_1+x_2+x_3, d=a+x_1+x_2+x_4,
e=a+x_1+x_2+x_5.
\end{equation*}
Substituting the above equalities into (\ref{hex}), (\ref{qq2}), and (\ref{qq4}) we obtain that
\begin{equation*}
2T_2+T_4=3+(3+6a)(x_2+x_3+x_5)+3x_3(x_2+x_5),
\end{equation*}
which is a sum of nonnegative numbers, hence, at least one of $T_2$ and $T_4$ must be nonnegative.
This completes the proof of Case 4.


\noindent{\emph{Case 5:}} $a \leq f \leq \min\{b,c,d,e\}$. Again, we study four different subcases.


\noindent{\emph{Subcase 5(i):}} $a \leq f \leq b \leq \min\{c, d, e\}$.

\noindent It follows that there exist nonnegative numbers
$x_i,\,i=1,\ldots,\,5$ such that
\begin{equation*}
f=a+x_1, b=a+x_1+x_2, c=a+x_1+x_2+x_3, d=a+x_1+x_2+x_4,
e=a+x_1+x_2+x_5.
\end{equation*}
Substituting the above equalities into (\ref{hex}), (\ref{qq2}), and (\ref{qq4}), we obtain that
\begin{equation*}
2T_2+T_4=3+(3+6a)(x_2+x_3+x_5)+3x_5(x_2+x_3),
\end{equation*}
which is certainly nonnegative numbers. It follows that either $T_2$ or $T_4$ is nonnegative.


\noindent{\emph{Subcase 5(ii)}:} $a \leq f \leq c \leq \min\{b, d, e\}$.

\noindent In this case there exist nonnegative numbers
$x_i,\,i=1,\ldots,\,5$ such that
\begin{equation*}
f=a+x_1, c=a+x_1+x_2, b=a+x_1+x_2+x_3, d=a+x_1+x_2+x_4,
e=a+x_1+x_2+x_5.
\end{equation*}
Substituting the above equalities into (\ref{hex}), (\ref{qq2}), (\ref{qq3}), and (\ref{qq4}), we obtain that
\begin{equation*}
T_2+T_3+T_4=3+(3+6a+3x_1)(x_1+x_2+x_5)+3x_1(x_1+x_2+2x_3)+3x_2x_5+3x_3x_4,
\end{equation*}
which is nonnegative. Hence, at least one of $T_2$, $T_3$, $T_4$ must be nonnegative.


\noindent{\emph{Subcase 5(iii):}} $a \leq f \leq d \leq \min\{b, c, e\}$.

\noindent It follows that there exist nonnegative numbers
$x_i,\,i=1,\ldots,\,5$ such that
\begin{equation*}
f=a+x_1, d=a+x_1+x_2, b=a+x_1+x_2+x_3, c=a+x_1+x_2+x_4,
e=a+x_1+x_2+x_5.
\end{equation*}
Substituting the above equalities into (\ref{hex}), (\ref{qq1}), (\ref{qq3}), and (\ref{qq4}) we obtain that
\begin{equation*}
T_1+T_3+T_4=3+(3+6a+6x_1)(2x_1+x_2+x_3)+6x_1(x_2+x_3)+3x_2x_3+3x_4x_5,
\end{equation*}
which is nonnegative. Hence, at least one of $T_1$, $T_3$, $T_4$ must be nonnegative.


\noindent{\emph{Subcase 5(iv):}} $a \leq f \leq e \leq \min\{b, c, d\}$

\noindent In this case there exist nonnegative numbers
$x_i,\,i=1,\ldots,\,5$ such that
\begin{equation*}
f=a+x_1, e=a+x_1+x_2, b=a+x_1+x_2+x_3, c=a+x_1+x_2+x_4,
d=a+x_1+x_2+x_5.
\end{equation*}
Substituting the above equalities into (\ref{hex}), (\ref{qq1}), and (\ref{qq3}) we obtain that
\begin{equation*}
2T_1+T_3=3+(3+6a+6x_1)(2x_1+x_2+x_3+x_5)+6x_1(x_2+x_3)+3x_3(x_2+x_5),
\end{equation*}
which is nonnegative, therefore, at least one of $T_1$ and $T_3$ must be nonnegative, which completes the analysis of Case 5, and with it, the proof of Theorem \ref{f64}.
\end{proof}

\section{\bf A Closer Look at $f_n (n-1)$}

As previously noticed, finding the largest $(n-1)$-gon contained in a
given convex $n$-gon $A_1 A_2 \ldots A_n$ is equivalent to finding
the smallest ear of the original $n$-gon. Define
\begin{equation}\label{g}
g_n:=\max \min_{1\leq k \leq n} \frac{\Delta(A_{k-1} A_k
A_{k+1})}{\Delta(A_1 A_2 \ldots A_n)}
\end{equation}
where the maximum is taken over all convex $n$-gons $A_1 A_2 \cdots
A_n$. Clearly, $g_n=1-f_n(n-1)$. Theorems (\ref{f54}) and
(\ref{f65}) can therefore be restated as $g_5 = (5-\sqrt{5})/{10}$
and $g_6 = 1/6$. In general,
\begin{equation}\label{lowerg}
g_n \geq \frac{4}{n} \cdot \sin^{2}\left(\frac{\pi}{n}\right)
\end{equation}
as it can easily be seen by checking the case when the $n$-gon is
regular. In this section we provide two upper bounds for $g_n$, the second of which shows that the order of magnitude of
the lower bound in (\ref{lowerg}) is the correct one.

We will need the following simple

\begin{lemma}\label{lemmadufengtan}
Let $ABCDE$ be a convex pentagon in which
\begin{equation*}
    \Delta(DEA) \leq \min(\Delta(CDE),\Delta(EAB))
\end{equation*} Then
\begin{equation*}
    \Delta(DEA) \leq \min(\Delta(ABD), \Delta(ACD))
\end{equation*}
\end{lemma}
\begin{proof}
Using equalities \eqref{laterlemma}, we have $c+d-1 \le \min(b+d-bc, a+c-ad)$ which imply that $c\le 1$ and $d\le 1$, that is, $c+d-1\le 1$. On the other hand, by \eqref{mama}, $\min(\Delta(ABD), \Delta(ACD))=\min(1+a, 1+b)>1$. The conclusion follows.
\end{proof}

The following theorem was proved by Du, Feng, and Tan \cite{dft}. We present a different argument below.
\begin{thm}\cite{dft}
For all $n\ge 4$ we have
\begin{equation}\label{gninequality}
g_{n+1}\le \frac{g_n}{1+g_n}.
\end{equation}
\end{thm}
\begin{proof}
Let $P=A_1A_2\ldots A_nA_{n+1}$ be an extremal convex $(n+1)$-gon for which
\begin{equation}
\frac{\max(\Delta(Q): Q\,\, \text{is an}\, n-\text{gon contained in}\,\, P)}{\Delta(P)}=f_{n+1}(n)=1-g_{n+1}.
\end{equation}
Without loss of generality we can assume that $Q=A_1A_2\ldots A_n$, that is, $A_nA_{n+1}A_1$ is the triangle of smallest area determined by three consecutive vertices of $P$. Denote $\Delta(A_nA_{n+1}A_1)=\alpha$.
Then,
\begin{equation*}
1-g_{n+1}=f_{n+1}(n)=\frac{\Delta(Q)}{\Delta(P)}=\frac{\Delta(Q)}{\alpha+\Delta(Q)}\implies \frac{\alpha}{\Delta(Q)}=\frac{g_{n+1}}{1-g_{n+1}}.
\end{equation*}
Let $R$ be the largest area $(n-1)$-gon contained in $Q$.
Clearly, $R$ is obtained by removing from $Q$ one of the following $n$ triangles: $A_1A_2A_3,A_2A_3A_4, \ldots, A_{n-2}A_{n-1}A_n,A_{n-1}A_nA_1$, or $A_nA_1A_2$.

Each of the first $n-2$ triangles has area $\ge \alpha$; this is because these triangles are ears of $P$ and we assumed that $A_{n}A_{n+1}A_1$ is the smallest area ear of $P$.

In particular, $\alpha=\Delta(A_nA_{n+1}A_1)\le \min(\Delta(A_{n-1}A_nA_{n+1}),\Delta(A_{n+1}A_1A_2))$.

Using Lemma \ref{lemmadufengtan} for the convex pentagon $ABCDE=A_1A_2A_{n-1}A_nA_{n+1}$ it follows that
\begin{equation*}
\alpha=\Delta(A_nA_{n+1}A_1)\le \min(\Delta(A_{n-1}A_nA_1),\Delta(A_nA_1A_2))
\end{equation*}
This proves that each of the triangles $A_1A_2A_3,A_2A_3A_4, \ldots, A_{n-2}A_{n-1}A_n,A_{n-1}A_nA_1,A_nA_1A_2$ have area $\ge \alpha$. It follows that $\Delta(R) \le \Delta(Q) - \alpha$ from which

\begin{equation*}
1-g_n=f_n(n-1)\le \frac{\Delta(R)}{\Delta(Q)}\le 1-\frac{\alpha}{\Delta(Q)}=1-\frac{g_{n+1}}{1-g_{n+1}}.
\end{equation*}
Comparing the first and last terms in the inequality chain above implies \eqref{gninequality}, as claimed.
\end{proof}

\begin{cor}
\begin{align}
&\text{(a) For every}\,\, n\ge 6,  g_n=1-f_n(n-1) \le \frac{1}{n}. \label{uppergn}\\
&\text{(b) For every}\,\, m, n\,\, \text{such that}\,\, 5\le m <n, f_n(m)\ge \frac{m}{n}.\label{lowerfmn}
\end{align}
\end{cor}
\begin{proof}
We proved in Theorem \ref{f65} then $f_6(5)=5/6$, so the first statement is true for $n=6$. Using inequality \eqref{gninequality}, it follows by induction that $g_n\le 1/n$, as claimed. In particular, $f_n(n-1)\ge (n-1)/n$ for all $n\ge 6$.

For proving the second part, given $5\le m<n$, let $P_n$ be an extremal convex $n$-gon for which
\begin{equation}
\frac{\max(\Delta(Q): Q\,\, \text{is an}\, m-\text{gon contained in}\,\, P_n)}{\Delta(P_n)}=f_{n}(m).
\end{equation}
Consider a finite sequence of convex polygons, $P_{n-1}, P_{n-2}, \ldots, P_{m+1}, P_m$, constructed as follows: for each $m\le k\le n-1$, $P_k$ is the convex $k$-gon obtained by removing the smallest area ear of $P_{k+1}$.
It follows that
\begin{equation*}
\frac{\Delta(P_k)}{\Delta(P_{k+1})}\ge f_{k+1}(k),
\end{equation*}
and therefore by using part (a) we obtain
\begin{equation*}
f_n(m)\ge \frac{\Delta(P_m)}{\Delta(P_n)}=\prod_{k=m}^{n-1}\frac{\Delta(P_k)}{\Delta(P_{k+1})}\ge \prod_{k=m}^{n-1} f_{k+1}(k)\ge \prod_{k=m}^{n-1}\frac{k}{k+1}=\frac{m}{n}.
\end{equation*}

In particular, $f_7(5)\ge 5/7=0.714\ldots$ improving the bound $f_7(5)\ge 5/(10-\sqrt{5})=0.644\ldots$ due to Du and Ding \cite{DD06}.
\end{proof}

Inequalities \eqref{lowerg} and \eqref{uppergn} provide bounds for $g_n$. However, the upper bound is rather weak as we expect that $g_n=\Theta(n^{-3})$.
We prove that this is indeed the case.

\begin{thm}\label{upperg}
For every $n\ge 4$
\begin{equation*}
g_n \leq \frac{4}{n} \cdot \sin^{2}\left(\frac{2\pi}{n}\right).
\end{equation*}
\end{thm}

\begin{proof}

By a celebrated result of John \cite {J48}, it is known that every
convex polygon lies between two concentric homothetic ellipses of
ratio $2$ - see figure 3.

After an appropriate affine transformation
$T$, these ellipses are mapped into a pair of concentric circles,
$C$ and $2C$. Since such a transformation preserves ratios, it would
be sufficient to prove that the result holds for the polygon
$T(P)=A_1A_2\ldots A_n$.

\vspace{0.5cm}

\begin{tikzpicture}[line cap=round,line join=round,>=triangle 45,x=1.0cm,y=1.0cm,scale=0.4]\label{fig3}
\clip(-1,-5.313299868620075) rectangle (40.522175265514385,15.617799346279181);
\fill[line width=2.pt,color=zzttqq,fill=zzttqq,fill opacity=0.10000000149011612] (7.367953226647206,13.300998288141466) -- (2.6145165728819166,11.024142075833725) -- (0.9368330480235795,6.0310363470886745) -- (1.9753990396025503,-0.5199183690248296) -- (8.007070759926572,-3.116333347972255) -- (12.960231642841665,1.3974342308132692) -- (13.639294021950992,7.269326567817447) -- (11.122768734663486,11.143976613323606) -- cycle;
\fill[line width=2.pt,color=zzttqq,fill=zzttqq,fill opacity=0.10000000149011612] (29.537342662275236,12.50210137154226) -- (24.664071471020065,10.664638463364081) -- (22.1076013379026,6.110926038748596) -- (23.665450325271056,-0.6796977523446712) -- (30.615853499684167,-2.9166091188224534) -- (37.04697367830779,2.3161656849023586) -- (37.24669790745763,7.1494920303275675) -- (34.57039323685028,10.664638463364083) -- cycle;
\draw [rotate around={90.:(7.,5.)},line width=2.pt] (7.,5.) ellipse (4.89667527342852cm and 3.8700683112066994cm);
\draw [rotate around={90.:(7.,5.)},line width=2.pt] (7.,5.) ellipse (9.793350546857042cm and 7.7401366224134005cm);
\draw [line width=2.pt,color=zzttqq] (7.367953226647206,13.300998288141466)-- (2.6145165728819166,11.024142075833725);
\draw [line width=2.pt,color=zzttqq] (2.6145165728819166,11.024142075833725)-- (0.9368330480235795,6.0310363470886745);
\draw [line width=2.pt,color=zzttqq] (0.9368330480235795,6.0310363470886745)-- (1.9753990396025503,-0.5199183690248296);
\draw [line width=2.pt,color=zzttqq] (1.9753990396025503,-0.5199183690248296)-- (8.007070759926572,-3.116333347972255);
\draw [line width=2.pt,color=zzttqq] (8.007070759926572,-3.116333347972255)-- (12.960231642841665,1.3974342308132692);
\draw [line width=2.pt,color=zzttqq] (12.960231642841665,1.3974342308132692)-- (13.639294021950992,7.269326567817447);
\draw [line width=2.pt,color=zzttqq] (13.639294021950992,7.269326567817447)-- (11.122768734663486,11.143976613323606);
\draw [line width=2.pt,color=zzttqq] (11.122768734663486,11.143976613323606)-- (7.367953226647206,13.300998288141466);
\draw [line width=2.pt] (29.537342662275236,5.431863659639269) circle (4.667397555544487cm);
\draw [line width=2.pt] (29.537342662275236,5.431863659639269) circle (9.334795111088974cm);
\draw [line width=2.pt,color=zzttqq] (29.537342662275236,12.50210137154226)-- (24.664071471020065,10.664638463364081);
\draw [line width=2.pt,color=zzttqq] (24.664071471020065,10.664638463364081)-- (22.1076013379026,6.110926038748596);
\draw [line width=2.pt,color=zzttqq] (22.1076013379026,6.110926038748596)-- (23.665450325271056,-0.6796977523446712);
\draw [line width=2.pt,color=zzttqq] (23.665450325271056,-0.6796977523446712)-- (30.615853499684167,-2.9166091188224534);
\draw [line width=2.pt,color=zzttqq] (30.615853499684167,-2.9166091188224534)-- (37.04697367830779,2.3161656849023586);
\draw [line width=2.pt,color=zzttqq] (37.04697367830779,2.3161656849023586)-- (37.24669790745763,7.1494920303275675);
\draw [line width=2.pt,color=zzttqq] (37.24669790745763,7.1494920303275675)-- (34.57039323685028,10.664638463364083);
\draw [line width=2.pt,color=zzttqq] (34.57039323685028,10.664638463364083)-- (29.537342662275236,12.50210137154226);
\draw (23.26600186697147,12.2) node[anchor=north west] {$A_1$};
\draw (20.150303892234554,7.189436876157535) node[anchor=north west] {$A_2$};
\draw (21.628263187943094,-0.9193668273244294) node[anchor=north west] {$A_3$};
\draw (37,8.307892559396425) node[anchor=north west] {$A_{n-2}$};
\draw (34.091055086890755,12.04204621737223) node[anchor=north west] {$A_{n-1}$};
\draw (28.65855605401613,14.119178200530172) node[anchor=north west] {$A_n$};
\draw (26.90098283749787,11.942873529922823) node[anchor=north west] {$l_n$};
\draw (23.505670941951234,9.266568859315475) node[anchor=north west] {$l_1$};
\draw (22.586939487862143,4.193573438910503) node[anchor=north west] {$l_2$};
\draw (26.221920458388546,-0.2) node[anchor=north west] {$l_3$};
\draw (30.735688037174075,12.062708067412705) node[anchor=north west] {$l_{n-1}$};
\draw (33.73155147442111,9.66601731761508) node[anchor=north west] {$l_{n-2}$};
\draw (27.460210679117317,2.91533837235177) node[anchor=north west] {$C$};
\draw (34.49050354519036,15.058571504659735) node[anchor=north west] {$2C$};
\begin{scriptsize}
\draw [fill=ududff] (7.367953226647206,13.300998288141466) circle (2.5pt);
\draw [fill=ududff] (2.6145165728819166,11.024142075833725) circle (2.5pt);
\draw [fill=ududff] (0.9368330480235795,6.0310363470886745) circle (2.5pt);
\draw [fill=ududff] (1.9753990396025503,-0.5199183690248296) circle (2.5pt);
\draw [fill=ududff] (8.007070759926572,-3.116333347972255) circle (2.5pt);
\draw [fill=ududff] (12.960231642841665,1.3974342308132692) circle (2.5pt);
\draw [fill=ududff] (13.639294021950992,7.269326567817447) circle (2.5pt);
\draw [fill=ududff] (11.122768734663486,11.143976613323606) circle (2.5pt);
\draw [fill=ududff] (29.537342662275236,12.50210137154226) circle (2.5pt);
\draw [fill=ududff] (24.664071471020065,10.664638463364081) circle (2.5pt);
\draw [fill=ududff] (22.1076013379026,6.110926038748596) circle (2.5pt);
\draw [fill=ududff] (23.665450325271056,-0.6796977523446712) circle (2.5pt);
\draw [fill=ududff] (30.615853499684167,-2.9166091188224534) circle (2.5pt);
\draw [fill=ududff] (37.04697367830779,2.3161656849023586) circle (2.5pt);
\draw [fill=ududff] (37.24669790745763,7.1494920303275675) circle (2.5pt);
\draw [fill=ududff] (34.57039323685028,10.664638463364083) circle (2.5pt);
\end{scriptsize}
\end{tikzpicture}
\vspace{-1cm}
\begin{figure}[h!]
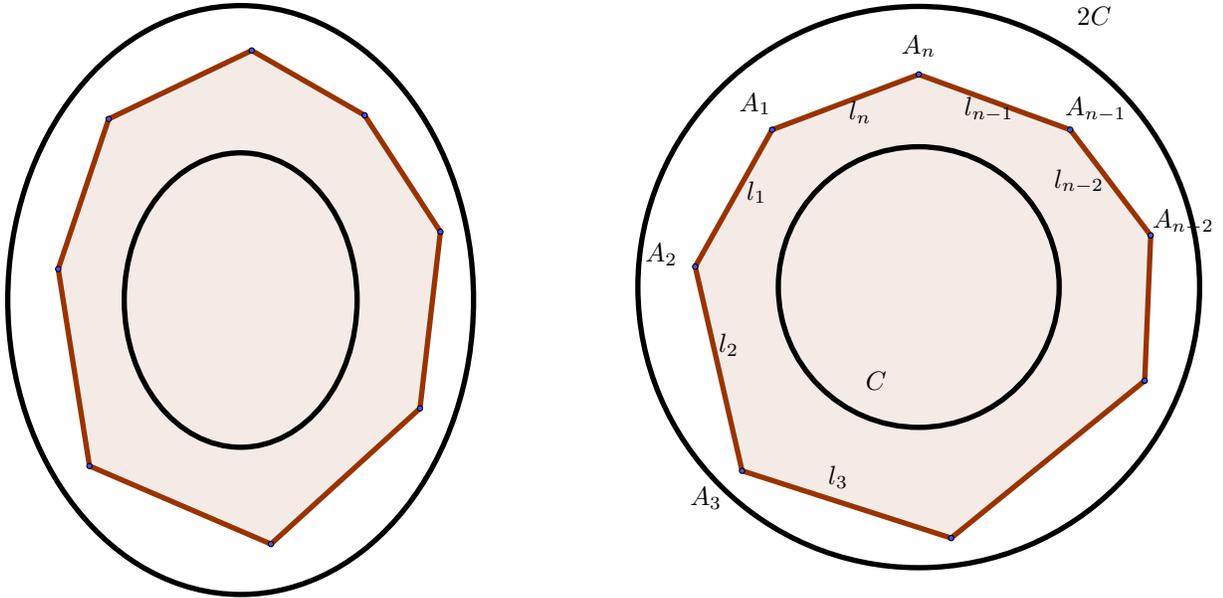

\caption{A convex polygon lying between two concentric circles, $C$ and $2C$}
\end{figure}

\noindent Denote the lengths of the sides $A_1A_2$, $A_2A_3\ldots$, $A_nA_1$ by
$l_1$, $l_2,\ldots$, $l_n$, respectively, and denote by
$L=l_1+l_2+\ldots l_n$, the perimeter of $A_1A_2\ldots A_n$. We have
that
\begin{eqnarray*}
\sum_{k=1}^n \sqrt[4]{2\Delta(A_{k-1}A_kA_{k+1})} &= &\sum_{k=1}^n \sqrt[4]{l_kl_{k+1}} \cdot \sqrt[4]{\sin(\angle{A_k})}\leq \\
&\leq&\left(\sum_{k=1}^n \sqrt{l_kl_{k+1}}\right) ^{1/2} \cdot \left(\sum_{k=1}^n \sqrt{\sin(\angle{A_k})}\right)^{1/2}\leq\\
&\leq&\left(\sum_{k=1}^n \frac{l_k + l_{k+1}}{2}\right)^{1/2} \cdot \left(n \sqrt{\sin\left(\frac{\sum_{k=1}^n\angle{A_k}}{n}\right)}\right)^{1/2}=\\
&=&\left(\sum_{k=1}^nl_k\right)^{1/2}\cdot\sqrt{n}\cdot\sqrt[4]{\sin\frac{(n-2)\pi}{n}}=L^{1/2}\cdot\sqrt{n}\cdot\sqrt[4]{\sin\left(\frac{2\pi}{n}\right)},
\end{eqnarray*}
\noindent where we used Cauchy-Schwarz inequality, the
geometric-arithmetic mean inequality and Jensen's inequality for the
concave function $\sqrt{\sin{x}}$ on the interval $(0,\,\pi)$. It
follows that
\begin{equation*}
\min_{1\le k\le n}\Delta(A_{k-1}A_kA_{k+1})\leq\frac{L^2\cdot \sin
\left(\frac{2\pi}{n}\right)}{2n^2},\quad\text{and after dividing
by}\,\, \Delta=\Delta(A_1A_2\ldots A_n)
\end{equation*}
\begin{equation}\label{A}
\min_{1\le k\le n}\frac{\Delta(A_{k-1}A_kA_{k+1})}
{\Delta}\leq\frac{L^2}{\Delta}\cdot\frac{sin\left(\frac{2\pi}{n}\right)}{2n^2}
\end{equation}
\noindent The polygon $A_1A_2...A_n$ contains a circle of radius
$r$, hence its area is at least as large as the area of the regular
$n$-gon circumscribed about a circle of radius $r$, that is,
\begin{equation}\label{Area}
\Delta \geq nr^2\tan\left(\frac{\pi}{n}\right).
\end{equation}
\noindent On the other hand, the polygon $A_1A_2...A_n$ is contained
in a circle of radius $2r$, hence its perimeter is no greater than
the perimeter of the regular $n$-gon inscribed in a circle of radius
$2r$, that is
\begin{equation}\label{perim}
L \leq 4nr\sin\left(\frac{\pi}{n}\right).
\end{equation}
Combining $(\ref {A}), (\ref {Area})$, and $(\ref {perim})$ we
obtain
\begin{equation*}
\min_{1\le k\le n} \frac{\Delta(A_{k-1}A_kA_{k+1})}{\Delta} \leq
\frac{16n^2r^2\sin^2\left(\frac{\pi}{n}\right)}{nr^2\tan\left(\frac{\pi}{n}\right)}
\cdot \frac{\sin\left(\frac{2\pi}{n}\right)}{2n^2}= \frac{4}{n}
\cdot \sin^2\left(\frac{2\pi}{n}\right).
\end{equation*}
This completes the proof.
\end{proof}



\begin{thebibliography}{99}

\bibitem{ABCG07} Ahn, H. K., Bae S. W., Cheong, O., Gudmundsson, J.:
Aperture-angle and Hausdorff-approximation of convex figures. \emph{Comput. Geom. (SCG'07)}, 37--45, ACM, New York (2007).

\bibitem{DD06} Du, Y., Ding, R.: On the maximum area pentagon in a planar point set, \emph{Applied Mathematics Letters},
{\bf 19}, 1228-1236 (2006).

\bibitem{DD08}Du, Y., Ding, R.: On maximum area polygons in a planar point set, \emph{Elemente der Mathematik}, {\bf 63}, 88-96 (2008).

\bibitem{dft} Du, Y.,  Feng, H., Tan, H.: More on the maximum area polygons in a planar point set, \emph{International Journal of Applied Mathematics}, {\bf 26}, 701-712 (2013).


\bibitem{LF39} Fejes, L.: \"{U}ber die Approximation konvexer Kurven durch
Polygonfolgen. (German) \emph{Compositio Math.} {\bf 6}, 456--467 (1939).


\bibitem{T53} Fejes T\'{o}th, L.: \emph{Lagerungen in der Ebene, auf der Kugel und im Raum}.
(German) Die Grundlehren der Mathematischen Wissenschaften in
Einzeldarstellungen mit besonderer Ber\"{u}cksichtigung der
Anwendungsgebiete, Band LXV. Springer-Verlag,
Berlin-G\"{o}ttingen-Heidelberg (1953).

\bibitem{TK93} Fejes T\'{o}th, G., Kuperberg, W.: Packing and covering with convex sets.
In: \emph{Handbook of convex geometry}, Vol. A, B, pp 799-860, North-Holland,
Amsterdam(1993).

\bibitem{FMRWY92}Fleischer, R., Mehlhorn, K., Rote, G., Welzl, E.,
Yap, C.K.: Simultaneous inner and outer approximation of shapes,
\emph{Algorithmica} {\bf 8}, 365-389 (1992).


\bibitem{J48}John, F.: Extremum problems with inequalities as subsidiary conditions.
Studies and Essays Presented to R. Courant on his 60th
Birthday, pp 187-204, Interscience Publishers, Inc., New York, NY (1948).

\bibitem{lazar47} L\'{a}z\'{a}r, D.: Sur l'approximation des courbes convexes par des
polygones. (French) \emph{Acta Universitatis Szegediensis} {\bf 11}, 129--132 (1947).




\bibitem{SAS40} S\'{a}s, E.: On a certain extremum-property of the ellipse. (German)
\emph{Compositio Math.} {\bf 7}, 474--476 (1940).

\bibitem{SFRW90}Schwarzkopf, O., Fuchs, U., Rote, G., Welzl,E.: Approximation
of convex figures by pairs of rectangles. In: Proc. Seventh Annual
Symposium on Theoretical Aspects of Computer Science, \emph{Lecture Notes
in Comput. Sci.}, vol. 415, 240-249 Springer, Berlin (1990).

\bibitem{Y87}Yap, C. K.: Algorithmic motion planning. In: \emph{Advances in Robotics}, Vol. 1,
eds. J.T. Schwartz and C.K. Yap, Chapter 3, Erlbaum, Hillsdale, NJ (1987).


\end{thebibliography}
\end{document}